\documentclass[11pt]{amsart}
\usepackage{verbatim}
\usepackage{amssymb}
\usepackage{amsthm}
\usepackage{amsmath}
\usepackage{amscd}
\usepackage{hyperref}
\usepackage{graphicx}
\usepackage[mathscr]{euscript}
\usepackage[all,knot,arc,cmtip, arrow]{xy}
\usepackage{tikz}
\usepackage{subfigure}

\allowdisplaybreaks

\theoremstyle{plain}
\newtheorem*{thm*}{Theorem}
\newtheorem*{propo*}{Proposition}
\newtheorem*{coro*}{Corollary}
\newtheorem{theorem}{Theorem}[section]

\newtheorem{prop}[theorem]{Proposition}
\newtheorem{cor}[theorem]{Corollary}

\theoremstyle{definition}
\newtheorem{definition}[theorem]{Definition}

\theoremstyle{remark}
\newtheorem{remark}[theorem]{Remark}

\newtheorem*{ack}{Acknowledgments}

\numberwithin{equation}{section}
\newcommand{\sO}{\ensuremath{\mathscr O}}
\newcommand{\C}{\ensuremath{\mathbb C}}
\newcommand{\R}{\ensuremath{\mathbb R}}
\newcommand{\Z}{\ensuremath{\mathbb Z}}

\DeclareMathOperator{\id}{id}
\DeclareMathOperator{\cl}{\sf cl}
\DeclareMathOperator{\tc}{{\sf TC}}
\DeclareMathOperator{\secat}{\sf secat}
\DeclareMathOperator{\cat}{\sf cat}

\DeclareMathOperator{\hdim}{hdim}
\DeclareMathOperator{\Conf}{Conf}
\newcommand{\cc}{{\mathsf c}}
\DeclareMathOperator{\b1}{{\mathbf 1}}

\newcommand{\Jo}{\mathop{\scalebox{1.5}{\raisebox{-0.2ex}{$\ast$}}}}

\begin{document}

\title[Parametrized topological complexity]
{Parametrized topological complexity of collision-free motion planning in the plane}
\author[D. Cohen]{Daniel C. Cohen}\thanks{D. Cohen was partially supported by an LSU Faculty Travel Grant}

\address{Department of Mathematics, Louisiana State University, Baton Rouge, LA 70803}
\email{\href{mailto:cohen@math.lsu.edu}{cohen@math.lsu.edu}}
\urladdr{\href{http://www.math.lsu.edu/~cohen/}
{www.math.lsu.edu/\char'176cohen}}

\author[M. Farber]{Michael Farber}\thanks{M. Farber was partially supported by EPSRC grant EP/V009877/1}
\address{School of Mathematical Sciences, Queen Mary University of London, E1 4NS London}
\email{\href{mailto:M.Farber@qmul.ac.uk}{M.Farber@qmul.ac.uk}}

\author[S. Weinberger]{Shmuel Weinberger}\thanks{S. Weinberger was partially supported by National Science Foundation grant DMS 1811071}
\address{Department of Mathematics, The University of Chicago, 5734 S University Ave, Chicago, IL 60637}
\email{\href{mailto:shmuel@math.uchicago.edu}{shmuel@math.uchicago.edu}}

\date{\today}   

\begin{abstract}
Parametrized motion planning algorithms have high degrees of universality and flexibility, as they are  designed to work under a variety of external conditions, which are viewed as parameters and form part of the input of the underlying motion planning problem. In this paper, we analyze the parametrized motion planning problem for the motion of many distinct points in the plane, moving without collision and avoiding multiple distinct obstacles with a priori unknown positions. This complements our 
prior work \cite{CFW}, where parametrized motion planning algorithms were introduced, and the obstacle-avoiding collision-free motion planning problem in three-dimensional space was fully investigated. The planar case 
requires different algebraic and topological tools than its spatial analog.
\end{abstract}

\keywords{parametrized topological complexity, obstacle-avoiding collision-free motion}

\subjclass[2010]{
55S40, %Sectioning fiber spaces and bundles 
55M30, %LS category of a space
55R80, %Discriminantal varieties, configuration spaces
70Q05%Control of mechanical systems
}

\maketitle

\section{Introduction} \label{sec:intro}
The goal of this paper is to give a topological
measurement of the complexity that robots must confront when navigating in 
a two-dimensional environment, avoiding impediments.

This work 
is a refinement of the work of Farber \cite{Fa03,Fa05} who
studied how much forking is necessary in the programming of a robotic
motion planner operating in a configuration space $X$ 
via a numerical invariant $\tc(X)$.  This,
in turn, was modeled on the seminal paper of Smale \cite{Sm}, which studied the
complexity of ``the fundamental theorem of algebra,'' 
that is, the amount of
forking that arises in the course of computation of solutions to polynomial
equations.  The invariant $\tc(X)$ also
measures the amount of instability that any motion planner must have,
that is, the number of different overlapping sets in a hybrid motion
planning system, or similarly how much forking arises in routing
algorithms.

Interesting as this invariant is, it only captures part of
the difficulty that a robot needs to negotiate.  A more realistic
theory would take into account the sensing capacity of the robot,
multiple robots that maneuver autonomously, energy, timing, and
communication.  We hope to investigate such issues in future work.  In
this paper and the previous one in this series \cite{CFW}, we focus on 
the problem of the computational complexity of
flexibly solving motion planning in a potentially changing
environment.

\renewcommand{\thefootnote}{\fnsymbol{footnote}}
A (point) robot\footnote[2]{
Of course, the idea of a point robot is an idealization.  The
difficulties confronted by a physical robot will only be greater.
Dealing with larger robots is related to the issue of dealing with
families of problems that need not form fibrations, 
and will not be addressed in this paper.}
moving around a convex room has a
simple task.  It can go from any point to any other along the straight
line connecting them.  If there is a single obstacle then any
algorithm must fork - the  one described would require a decision
about whether to go around the obstacle to the left or the right. 
It turns out that two obstacles are harder than one, but
then it gets no harder.  

Similarly, the complexity of motion in a
graph can only have three values - trivial for a tree, complexity $1$ for a
graph with a single cycle, but only 
increasing by one for any graph with more than $1$ cycle.
The reason for this is that any connected graph can be
described as a union of two trees, so if there is a specific graph that
needs to be navigated, one can make use of such a decomposition.  (A
similar statement can be made regarding the part of a room that is
complementary to any union of a finite number of convex subsets.)

Here we shall see that if the robot each day needs to move
around the room where the obstacles have also been moved around, the
complexity of the problem to be solved indeed grows.  
More generally, our main result provides a solution to 
the analogous problem for an arbitrary finite number of robots that
are centrally controlled.  The predecessor paper \cite{CFW} studies the
three-dimensional version of this problem, 
for example, 
for submarines navigating a mined part of the ocean.
Interestingly, the mathematics is somewhat more difficult in this
two-dimensional situation than the three-dimensional case.

In both cases, however, the general formalism is the same.
We consider a parameter space that describes the possible location of
obstacles, and therefore study a parametrized form of topological
complexity.  In our situation we have 
the mathematical structure of a (Hurewicz) fibration 
describing  the
set of motion planning problems, which enables the application of the
powerful apparatus of algebraic topology.  Some of the other problems
mentioned above require a weakening of this hypothesis, and cannot be
directly approached by the methodology of this paper.

\subsection*{Parametrized motion planning }  An autonomously functioning system in robotics typically includes a motion planning algorithm which takes as input the initial and terminal states of the system,  and produces as output  a motion of the system from the initial state to the terminal state. The theory of robot motion planning algorithms is an active area in the field of robotics, see \cite{Lat,Lav} and the references therein. A topological approach to the robot motion planning problem was developed in \cite{Fa03,Fa05}, where topological techniques clarify  
relationships between instabilities occurring in robot 
motion planning algorithms and topological features of the 
configuration spaces of the relevant autonomous systems. 

In a recent article \cite{CFW}, we 
developed a new approach to the theory of motion planning algorithms. In this ``parametrized'' approach, algorithms are required to be \emph{universal}, so that they are able to function under a variety of situations, involving different external conditions which are viewed as parameters and are part of the input of the underlying motion planning problem. Typical situations of this kind arise when one is dealing with the collision-free motion of many objects (robots) moving in two- or
three-dimensional space avoiding a set of obstacles, and the positions of the obstacles are a priori unknown. 

In the current paper, we continue our investigation of the problem of collision-free motion of many particles avoiding multiple moving obstacles, focusing primarily on the planar case. A team of robots moving in an obstacle-filled room is one example. As another illustration, consider a spymaster coordinating  the motion of a team of spies in a planar theatre of operations each day. Spies must avoid opposition checkpoints, which may be repositioned daily, and may not meet so as to avoid potentially compromising one another. The analogous problem in three-dimensional space, for instance, maneuvering a submarine fleet in waters infested with repositionable mines, was analyzed in \cite{CFW}. 

In each of these motion planning problems, one is faced with a space of allowable configurations of the robots/spies/submarines which depends on parameters, the daily positions of the obstacles/checkpoints/mines.   A motion planning algorithm should then be flexible enough to deal with changes in the parameters. 
The algebraic and topological tools used to analyze the complexity of such algorithms in the planar and spatial cases are essentially different. These differences are reflected by a numerical invariant, the \emph{parametrized topological complexity}, which differs in the planar and spatial cases.

\subsection*{Parametrized topological complexity} 
We reformulate these considerations mathematically, using the language of algebraic topology. 

Let $X$ be a path-connected topological space. Viewing $X$ as the space of all states of a mechanical system, the motion planning problem from robotics takes as input an initial state and a terminal state of the system, and requests as output a continuous motion of the system from the initial state to the terminal state. That is, given $(x_0,x_1) \in X\times X$, one would like to produce a continuous path $\gamma\colon I \to X$ with $\gamma(0)=x_0$ and $\gamma(1)=x_1$, where $I=[0,1]$ is the unit interval. 

Let $X^I$ be the space of all continuous paths in $X$, equipped with the compact-open topology. The map $\pi\colon X^I \to X\times X$, $\pi(\gamma)=(\gamma(0),\gamma(1))$, is a fibration, with fiber $\Omega X$, the based loop space of $X$. A solution of the motion planning problem, a motion planning algorithm, is then a section of this fibration, a map $s\colon X\times X \to X^I$ with $\pi\circ s=\id_{X\times X}$. If $X$ is not contractible, the section $s$ cannot be globally continuous, see \cite{Fa05}. 

The topological complexity of $X$ is defined to be the sectional category, or Schwarz genus, of the fibration $\pi\colon X^I \to X\times X$, $\tc(X)=\secat(\pi)$. That is, $\tc(X)$ is the smallest number $k$ for which there is an open cover $X\times X=U_0\cup U_1\cup \dots \cup U_k$ and the map $\pi$ admits a continuous section $s_j\colon U_j \to X^I$ satisfying $\pi\circ s_j = \id_{U_j}$ for each $j$. 
The numerical homotopy type invariant $\tc(X)$ provides a measure of the navigational complexity in $X$. 
Significant recent advances in the subject include work of Dranishnikov \cite{Dr} on the topological complexity of spaces modeling hyperbolic groups, and work of Grant and Mescher \cite{GrM} on the topological complexity of symplectic manifolds. We refer to the surveys \cite{dc18,Fa18} and recent work of Ipanaque Zapata and Gonz\'alez \cite{IG} for discussions of topological complexity and motion planning algorithms in the context of collision-free motion.

A parametrized approach to the motion planning problem was recently put forward in \cite{CFW}. 
In the parametrized setting, constraints are imposed by external conditions encoded by an auxiliary topological space $B$, and the initial and terminal states of the system, as well as the motion between them, must satisfy the same external conditions. 

This is modeled by a fibration $p\colon E \to B$, with nonempty path-connected fibers. For $b\in B$, the fiber $X_b=p^{-1}(b)$ is viewed as the space of achievable configurations of the system given the constraints imposed by $b$. Here, a motion planning algorithm takes as input initial and terminal (achievable given $b$) states of the system, and produces a continuous (achievable given $b$) path between them. That is, the initial and terminal points, as well as the path between them, all lie within the same fiber $X_b$. The parametrized topological complexity of the fibration $p\colon E \to B$ is then defined to be the sectional category of the associated fibration $\Pi\colon E^I_B \to E\times_BE$, where $E\times_BE$ is the space of all pairs of configurations lying in the same fiber of $p$, $E^I_B$ is the space of paths in $E$ lying in the same fiber of $p$, and the map $\Pi$ sends a path to its endpoints.

\subsection*{Obstacle-avoiding, collision-free motion} 
Investigating the collision-free motion of $n$ distinct ordered particles in a topological space $Y$ leads one to study the standard (unparametrized) topological complexity of the classical configuration space
\[
\Conf(Y,n)=\{(y_1,y_2,\dots,y_n) \in Y^n \mid y_i \neq y_j\ \text{for}\ i \neq j\}
\]
of $n$ distinct ordered points in $Y$. Similarly, investigating the collision-free motion of $n$ distinct particles in a manifold $Y$ in the presence of $m$ distinct obstacles, with a priori not known positions, leads one to study the parametrized topological complexity of the classical Fadell-Neuwirth bundle, the locally trivial fibration
\[
p\colon \Conf(Y,m+n) \to \Conf(Y,m),\quad p(y_1,\dots,y_m,y_{m+1},\dots,y_{m+n}) =(y_1,\dots,y_m),
\]
with fiber $p^{-1}(y_1,\dots,y_m)=\Conf(Y\smallsetminus \{y_1,\dots,y_m\},n)$.

In this paper, we  complete the determination of the parametrized topological complexity of the Fadell-Neuwirth bundles of Euclidean configuration spaces begun in \cite{CFW}. Our main result, Theorem \ref{thm:main}, includes the following as a special case.

\begin{thm*} For positive integers $m$ and $n$, 
 the parametrized topological complexity of the motion of $n$ non-colliding 
 particles
 in the plane $\R^2$, in the presence of $m$ non-colliding point obstacles with a priori unknown positions is equal to $2n+m-2$.
\end{thm*}
The case $m=1$ of this result reduces to the previously known determination of the (standard) topological complexity of $\Conf(\R^2\smallsetminus\{0\}),n)$, see Remark \ref{rem:m=1}.

As discussed in Section \ref{sec:config pTC}, different techniques yield the same parametrized topological complexity for obstacle-avoiding collision-free motion in $\R^d$ for any $d\ge 4$ even.  
The analogous motion planning problem in $\R^d$, for $d \ge 3$ odd, 
was analyzed in \cite[Thm.~9.1]{CFW}, where it was shown that the parametrized topological complexity is $2n+m-1$. 
These results provide examples of fibrations for which the parametrized topological complexity exceeds the (standard) topological complexity of the fiber, since $\tc(\Conf(\R^d\smallsetminus \{y_1,\dots,y_m\},n))=2n$ as shown in \cite{FGY}.

Our main result also illustrates that parametrized topological complexity may differ significantly from other notions of the topological complexity of a map which appear in the literature. If $p\colon E \to B$ is a fibration which admits a (homotopy) section, as is the case for many Fadell-Neuwirth bundles, then the topological complexity of $p$, as defined in either \cite{MW} or \cite{P19}, is equal to $\tc(B)$. 
For the Fadell-Neuwirth bundle $p\colon \Conf(\R^d,m+n) \to \Conf(\R^d,m)$ with $d\ge 2$ even, we have $\tc(B)=\tc(\Conf(\R^d,m))=2m-2$ (see, for instance, \cite{Fa18}), which differs from 
the parametrized topological complexity of the bundle 
unless the number of obstacles is twice the number of robots.

\section{Parametrized topological complexity} \label{secpTC}

In this brief section, we 
record 
requisite material from \cite{CFW}. Recall the broad framework: We wish to analyze the complexity of a motion planning algorithm in an environment which may change under the influence of external conditions. These conditions, parameters treated as part of the input of the algorithm, are encoded by a topological space $B$. Associated to each choice of conditions, that is, to each point $b \in B$, one has a configuration space $X_b$ of achievable configurations in which motion planning must take place. The motion planning algorithim must thus be sufficiently flexible so as to adapt to different external conditions, that is, different points in the parameter space $B$.

Let $p\colon E \to B$ be a Hurewicz fibration (briefly, a fibration), a continuous map which has the homotopy lifting property with respect to every space \cite[Sec. 2.2]{Spa}. We assume throughout that $p$ has nonempty, path-connected fiber $X$. Let $E^I_B$ denote the space of all continuous paths $\gamma\colon I \to E$ which lie in a single fiber of $p$, so that $p\circ\gamma$ is a 
constant path in $B$. Let 
\[
E\times_BE=\{(e,e') \in E \times E \mid p(e)=p(e')\}
\]
be the space of pairs of points in $E$ which lie in the same fiber. 

\begin{prop} \label{prop:Pi} If $p\colon E \to B$ is a  Hurewicz fibration, then 
the map 
\[
\Pi \colon E^I_B \to E\times_BE, \qquad \gamma \mapsto (\gamma(0),\gamma(1))
\]
given by sending a path to its endpoints is a Hurewicz fibration.
\end{prop} 
The  fiber of  $\Pi \colon E^I_B \to E\times_BE$ consists of 
all paths $\gamma \subset X_b=p^{-1}(b)$, starting and ending at $b=p(e)=p(e')$. In other words, the fiber  $\Pi^{-1}(e,e')$ is the space $\Omega X$ of based loops in $X_b=X$. The fact that $\Pi$ is a fibration is a consequence of a more general result which, for the sake of completeness, is stated and proved in the Appendix.

\begin{definition} \label{def:pTC}
The parametrized topological complexity $\tc[p\colon E \to B]$ of the fibration $p\colon E \to B$ is the sectional category of the fibration $\Pi \colon E^I_B \to E\times_BE$,
\[
\tc[p\colon E \to B] :=\secat(\Pi\colon E^I_B \to E\times_BE).
\]
That is, $\tc[p\colon E \to B]$ is equal to the smallest nonnegative integer $k$ for which the space $E\times_BE$ admits an open cover 
\[
E\times_BE = U_0\cup U_1 \cup \dots \cup U_k,
\]
and the map $\Pi \colon E^I_B\to E\times_BE$ admits a continuous section $s_i\colon U_i \to E^I_B$ for each $i$, $0\le i\le k$.

If the fibration $p$ is clear from the context, we sometimes use the abbreviated notation $\tc[p\colon E \to B]=\tc_B(X)$, to emphasize the role of the fiber $X$. 
\end{definition}

As shown in \cite[Prop.~5.1]{CFW}, parametrized topological complexity is an invariant of fiberwise homotopy equivalence.

For a topological space $Y$, let $\dim(Y)$ denote the covering dimension of $Y$, and let $\hdim(Y)$ denote the homotopy dimension of $Y$, the minimal dimension of a space $Z$ homotopy equivalent to $Y$. Since the parametrized topological complexity of $p\colon E \to B$ is defined to be the sectional category of the associated fibration $\Pi\colon E^I_B \to E\times_BE$, we have
\[
\tc[p\colon E \to B] \le \cat(E\times_BE) \le \hdim(E\times_BE),
\]
where $\cat(Y)$ is the Lusternik-Schnirelmann category of $Y$ (cf. \cite{Sch}). We also have the following.

\begin{prop}[{\cite[Prop.~7.1]{CFW}}] \label{prop:upper}
Let $p\colon E \to B$ be a locally trivial fibration of metrizable topological spaces, with path-connected fiber $X$. Then,
\[
\tc_B(X)=\tc[p\colon E \to B]  \le 2\dim(X)+\dim(B).
\]
\end{prop}

Parametrized topological complexity admits a cohomological lower bound. 
For a graded ring $A$, let $\cl(A)$ denote the cup length of $A$, the largest integer $q$ for which there are homogeneous elements $a_1,\dots, a_q$ of positive degree in $A$ such that 
$a_1\cdots a_q \neq 0$. 

\begin{prop}[{\cite[Prop.~7.3]{CFW}}] \label{prop:cup}
Let $p\colon E \to B$ be a fibration with path-connected fiber, and let $\Delta\colon E \to E\times_B E$ be the diagonal map, $\Delta(e)=(e,e)$. Then the parametrized topological complexity of $p\colon E \to B$ is greater than or equal to the cup length of the kernel of the map in cohomology induced by $\Delta$, 
\[
\tc[p\colon E \to B] \ge \cl\left(\ker\bigl[\Delta^*\colon H^*(E\times_B E; R) \to H^*(E;\Delta^*R)\bigr]\right),
\]
for any commutative coefficient ring $R$.
\end{prop}

We conclude this section by recording 
a 
product inequality for parametrized topological complexity, which we will make use of in Section \ref{sec:config pTC} below. 

\begin{prop}[{\cite[Prop.~6.1]{CFW}}] \label{prop:product}
Let $p'\colon E' \to B'$ and $p''\colon E'' \to B''$ be fibrations with path-connected fibers $X'$ and $X''$ respectively. 
Let $B=B'\times B''$, $E=E'\times E''$, $X=X'\times X''$, and $p=p'\times p''$. Then the product fibration $p\colon E \to B$ satisfies
\[
\tc[p\colon E \to B] \le \tc[p'\colon E' \to B']  + \tc[p''\colon E'' \to B''].
\]
Equivalently, in abbreviated notation, 
\[
\tc_{B'\times B''}(X'\times X'') \le \tc_{B'}(X')  + \tc_{B''}(X'').
\]
\end{prop}

\section{Cohomology of the obstacle-avoiding configuration space}
In this section, we study the structure of the cohomology rings of configuration spaces arising in the context of our main theorem. 
Let $E=\Conf(\R^d,m+n)$ and $B=\Conf(\R^d,m)$. Then, the Fadell-Neuwirth bundle of configuration spaces is $p\colon E \to B$, with fiber $X= \Conf(\R^d\smallsetminus\sO_m,n)$, where $\sO_m$ is a set of $m$ distinct points in $\R^d$. In order to utilize 
Proposition \ref{prop:cup} subsequently, we analyze the cohomology ring of the ``obstacle-avoiding configuration space'' $E\times_BE$.

We use homology and cohomology with integer coefficients, and suppress the coefficients, throughout.  
The principal objects of study, $E$, $B$, $X$, and $E\times_BE$, all have torsion free integral homology and cohomology.  This is well known for the classical configuration spaces, see \cite{FH}. 
We first recall several results from \cite{CFW}. While the focus of \cite{CFW} is mainly on odd dimensions $d\ge 3$, it is readily checked that these specific results hold for any dimension $d\ge 2$.

\begin{prop}[{\cite[Prop. 9.2]{CFW}}] \label{prop:HEBE}
Let $p\colon E=\Conf(\R^d,m+n) \to B=\Conf(\R^d,m)$ be the Fadell-Neuwirth bundle of configuration spaces. 
Then, the cohomology ring $H^*(E\times_B E)$ contains degree $d-1$ elements $\omega^{}_{i,j}$ and $\omega'_{i,j}$, $1\le i< j \le m+n$, which satisfy the relations
\begin{equation*} \label{eq:Hrel}
\begin{array}{ll}
\omega'_{i,j}=\omega^{}_{i,j}\ \text{for}\ 1\le i<j \le m, \quad&
\omega^{}_{i,j}\omega^{}_{i,k}-\omega^{}_{i,j}\omega^{}_{j,k}+\omega^{}_{i,k}\omega^{}_{j,k}=0 \ \text{for}\ i<j <k,\\[4pt]
(\omega^{}_{i,j})^2=(\omega'_{i,j})^2=0 \ \text{for}\ i<j,&
\omega'_{i,j}\omega'_{i,k}-\omega'_{i,j}\omega'_{j,k}+\omega'_{i,k}\omega'_{j,k}=0\ \text{for}\ i<j <k. 
\end{array}
\end{equation*}
\end{prop}

Since $\omega'_{i,j}=\omega^{}_{i,j}$ for $1\le i<j\le m$, the last of these relations may be expressed as $\omega^{}_{i,j}\omega'_{i,k}-\omega^{}_{i,j}\omega'_{j,k}+\omega'_{i,k}\omega'_{j,k}=0$ for such $i$ and $j$. We refer to relations of this general form as ``three term relations''.

For a natural numbers $p \le q$, let $[q]=\{1,2,\dots,q\}$ and $[p,q]=\{p,p+1,\dots,q\}$. Let $I=(i_1,\dots,i_\ell)$ and $J=(j_1,\dots,j_\ell)$ be sequences of elements in $[m+n]$. 
If $i_k <j_k$ for each $k$, $1\le k\le \ell$, we write $I<J$ and define cohomology classes
\[
\omega^{}_{I,J} = \omega^{}_{i_1,j_1}\omega^{}_{i_2,j_2} \cdots \omega^{}_{i_\ell,j_\ell} \quad\text{and}\quad
\omega'_{I,J} = \omega'_{i_1,j_1}\omega'_{i_2,j_2} \cdots \omega'_{i_\ell,j_\ell} 
\]
in $H^{(d-1)\ell}(E\times_BE)$. If $\ell=0$, set $\omega^{}_{I,J}=\omega'_{I,J}=1$. 

For a sequence $J$, write $J \subset [q]$, respectively, $J \subset [p,q]$, to communicate that the elements of $J$ are in the set $[q]$, respectively, $[p,q]$. 
Call the sequence $J=(j_1,j_2,\dots,j_\ell)$ increasing if $j_1<j_2<\dots<j_\ell$. 

\begin{prop}[{\cite[Prop.~9.3]{CFW}}] \label{prop:basis}
The cohomology of $E\times_B E$ is torsion free. A basis for $H^*(E\times_BE)$ is given by the set of cohomology classes
\[
\omega^{}_{I_1,J_1}\omega^{}_{I_2,J_2}\omega'_{I_3,J_3}, 
\]
where $J_1 \subset [m]$, $J_2,J_3 \subset [m+1,m+n]$ are increasing sequences, and $I_1$, $I_2$, and $I_3$ are sequences with $I_1<J_1$, $I_2<J_2$, and $I_3<J_3$.
\end{prop}

\begin{remark} \label{rem:ring}
As noted in \cite{CFW}, one can use Proposition \ref{prop:basis} to show that the cohomology ring $H^*(E\times_B E)$ is generated as a ring by the classes $\omega^{}_{i,j}$ and $\omega'_{i,j}$, $1\le i< j \le m+n$, and that the relations recorded in Proposition \ref{prop:HEBE}
are, in fact, a defining set of relations. Briefly, let $\mathcal R$ denote the graded commutative ring generated by $\omega^{}_{i,j}, \omega'_{i,j}$ for $1\le i< j \le m+n$, and let $\mathcal I$ be the ideal $\mathcal R$ generated by 
\[
\left\{
\begin{matrix}
\omega'_{i,j}-\omega^{}_{i,j}\ \text{for}\ 1\le i<j \le m, \,&
\omega^{}_{i,j}\omega^{}_{i,k}-\omega^{}_{i,j}\omega^{}_{j,k}+\omega^{}_{i,k}\omega^{}_{j,k}=0 \ \text{for}\ i<j <k,\\[4pt]
(\omega^{}_{i,j})^2,\ (\omega'_{i,j})^2 \ \text{for}\ i<j, \hfill&
\omega'_{i,j}\omega'_{i,k}-\omega'_{i,j}\omega'_{j,k}+\omega'_{i,k}\omega'_{j,k}=0\ \text{for}\ i<j <k
\end{matrix}
\right\}.
\]
One can then check that any monomial $\mu$ in $\mathcal R$ may be expressed as $\mu = \alpha + \beta$, where $\alpha \in {\mathcal I}$ and $\beta$ is a linear combination of the (homogeneous) basis elements recorded in Proposition \ref{prop:basis}. This may be used to show that $H^*(E\times_B E) \cong {\mathcal R}/{\mathcal I}$ as asserted. We will not make use of the full ring structure of $H^*(E\times_B E)$ in what follows.

The space $E\times_B E$ may be realized as the complement of an arrangement of subspaces in $(\R^d)^{m+2n}$, and the ring structure of  $H^*(E\times_B E)$ may also be obtained using the theory of hyperplane and subspace arrangements, see, for instance, \cite{dLS,OT}.
\end{remark}

We conclude this section with a technical result which will be used in the proof of the main theorem. 
For a sequence $J=(j_1,\dots,j_\ell)$, let $\widehat{J}=(j_1,\dots,j_{\ell-1})$.
\begin{definition} \label{def:adm}
Let $J=(j_1,\dots,j_\ell)$ be an increasing sequence. A $J$-\emph{admissible} 
sequence $I=(i_1,\dots,i_\ell)$ is defined recursively as follows.
If $|J|=\ell=1$, then $I$ is $J$-admissible if and only if $I=J$. If $|J|=\ell \ge 2$, then $I$ is $J$-admissible if
\begin{enumerate} 
\item $I$ is nondecreasing, $i_1\le \dots \le i_\ell$,
\item $\widehat{I}=(i_1,\dots,i_{\ell-1})$ is $\widehat{J}$-admissible, and 
\item either $i_{\ell}=i_{\ell-1}$ or $i_\ell=j_\ell$.
\end{enumerate} 
\end{definition}
For instance, if $J=(j_1,j_2)$, the $J$-admissible sequences are $(j_1,j_1)$ and $J$ itself.

\begin{prop} \label{prop:rewrite}
If $J=(j_1,\dots,j_\ell)$ is an increasing sequence and $r > j_\ell$, 
then
\begin{align*}
\omega^{}_{j_1,r}\omega^{}_{j_2,r}\cdots \omega^{}_{j_\ell, r} &=
(-1)^\ell \sum_I (-1)^{d_I} \omega^{}_{i_1,j_2} \omega^{}_{i_2,j_3} \cdots \omega^{}_{i_{\ell-1},j_\ell} \omega^{}_{i_\ell,r},\\
\intertext{and}
\omega'_{j_1,r}\omega'_{j_2,r}\cdots \omega'_{j_\ell, r} &=
(-1)^\ell \sum_I (-1)^{d_I} \omega'_{i_1,j_2} \omega'_{i_2,j_3} \cdots \omega'_{i_{\ell-1},j_\ell} \omega'_{i_\ell,r},
\end{align*}
where the sums are over all $J$-admissible sequences $I$, and $d_I$ is the number of distinct elements in $I$.
\end{prop}
Observe that the sums above are linear combinations of distinct elements of the basis 
for $H^*(E\times_BE)$ given in Proposition \ref{prop:basis}.
\begin{proof} 
Let $R = (r,r,\dots,r)$ be the constant sequence of length $\ell$. 
The proposition asserts that  
\[
\omega^{}_{J,R}=(-1)^\ell \sum_I (-1)^{d_I} \omega^{}_{I,K} \quad \text{and} \quad \omega'_{J,R}=(-1)^\ell \sum_I (-1)^{d_I} \omega'_{I,K},
\]
where  $K=(j_2,\dots,j_\ell,r)$.  Clearly, it suffices to consider $\omega^{}_{J,R}$. 

The proof is by induction on $\ell=|J|$, with the case $\ell=1$ trivial. 
The case $\ell=2$ is the three term relation
$\omega^{}_{j_1,r}\omega^{}_{j_2,r}=\omega^{}_{j_1,j_2}\omega^{}_{j_2,r}-\omega^{}_{j_1,j_2}\omega^{}_{j_1,r}$, which will be crucial subsequently.

Assume that $\ell \ge 3$. For $J=(j_1,\dots,j_\ell)$, recall that $\widehat{J}=(j_1,\dots,j_{\ell-1})$, and let $\widehat{R}$ be the constant sequence of length $\ell-1$.
By induction, we have
\[
\omega^{}_{\widehat{J},\widehat{R}} 
=(-1)^{\ell-1}\sum_{\widehat{I}} (-1)^{d_{\widehat{I}}} \omega^{}_{i_1,j_2}  \cdots \omega^{}_{i_{\ell-2},j_{\ell-1}} \omega^{}_{i_{\ell-1},r},
\]
where the sum is over all $\widehat{J}$-admissible sequences $\widehat{I}=(i_1,\dots,i_{\ell-1})$. Since $\omega^{}_{J,R}=\omega^{}_{\widehat{J},\widehat{R}} \omega^{}_{j_\ell,r}$, we obtain
\[
\begin{aligned}
\omega^{}_{J,R} &= (-1)^{\ell-1}\sum_{\widehat{I}} (-1)^{d_{\widehat{I}}} \omega^{}_{i_1,j_2}  \cdots \omega^{}_{i_{\ell-2},j_{\ell-1}} \omega^{}_{i_{\ell-1},r}\omega^{}_{j_\ell,r}\\
 &= (-1)^{\ell-1}\sum_{\widehat{I}} (-1)^{d_{\widehat{I}}} \omega^{}_{i_1,j_2}  \cdots \omega^{}_{i_{\ell-2},j_{\ell-1}} (\omega^{}_{i_{\ell-1},j_\ell}\omega^{}_{j_\ell,r} - \omega^{}_{i_{\ell-1},j_\ell}\omega^{}_{i_{\ell-1},r}),
\end{aligned}
\]
using the three term relations on the second line. For $\widehat{I}$ as above, 
let $P=(i_1,\dots,i_{\ell-1},j_\ell)$ and $Q=(i_1,\dots,i_{\ell-1},i_{\ell-1})$. Note that 
$d_P=d_{\widehat{I}}+1$ and $d_Q=d_{\widehat{I}}$. Further, as is clear from Definition \ref{def:adm}, every $J$-admissible sequence $I$ arises from a $\widehat{J}$-admissible sequence $\widehat{I}$ by adjoining either $j_\ell$ or $i_{\ell-1}$.  Thus, 
\[
\begin{aligned}
\omega^{}_{J,R} &= (-1)^{\ell-1}\sum_{\widehat{I}} (-1)^{d_{\widehat{I}}} \omega^{}_{P,K}+(-1)^{\ell}\sum_{\widehat{I}} (-1)^{d_{\widehat{I}}} \omega^{}_{Q,K}\\
&=(-1)^\ell \sum_{P} (-1)^{d_P} \omega^{}_{P,K}+(-1)^\ell \sum_{Q} (-1)^{d_Q} \omega^{}_{Q,K}\\
&=(-1)^\ell \sum_{I} (-1)^{d_I} \omega^{}_{I,K},
\end{aligned}
\]
where the last sum is over all $J$-admissible sequences as required.
\end{proof}

\section{Obstacle-avoiding collision-free motion in 
even dimensions} \label{sec:config pTC}
We now state and prove our main theorem, 
determining the parametrized topological complexity of obstacle-avoiding collision-free motion in any Euclidean space $\R^d$ of positive even dimension. The case $d=2$ of the plane was highlighted in the Introduction.

\begin{theorem} \label{thm:main}
For positive integers $n$, $m$, and $d\ge 2$ even, the parametrized topological complexity of the motion of $n$ non-colliding particles in $\R^d$ in the presence of $m$ non-colliding point obstacles with a priori unknown positions is equal to $2n+m-2$. In other words, 
the parametrized topological complexity of the Fadell-Neuwirth bundle 
$p\colon \Conf(\R^d,m+n) \to \Conf(\R^d,m)$ is
\begin{equation*} \label{eq:pTCFN}
{\tc}\bigl[p\colon \Conf(\R^d,m+n) \to \Conf(\R^d,m)\bigr]=2n+m-2.
\end{equation*}
\end{theorem}

Let $E= \Conf(\R^d,m+n)$ and  $B=\Conf(\R^d,m)$, so that the Fadell-Neuwirth bundle is $p\colon E \to B$. The fiber of this bundle is $X=\Conf(\R^d\smallsetminus\sO_m,n)$, where $\sO_m$ is a set of $m$ distinct points (representing the obstacles). Each of the spaces $E$, $B$, $X$, and $E\times_BE$ has the homotopy type of a finite CW-complex of known dimension. For the configuration spaces  $B$, $E$, and $X$, see \cite{FH}. For $E\times_BE$, this can be shown using various forms of Morse theory, cf. \cite{Adi,GM}. The dimensions of these CW-complexes are
\begin{equation} \label{eq:hdims}
\begin{array}{ll}
\hdim B = (m-1)(d-1),\quad &
\hdim E = (m+n-1)(d-1),\\
\hdim X = n(d-1),\quad &
\hdim E\times_BE =(2n+m-1)(d-1). 
\end{array}
\end{equation}
Furthermore, each of the spaces  $E$, $B$, $X$, and $E\times_BE$ is $(d-2)$-connected, as each is obtained removing codimension $d$ subspaces from a Euclidean space.

\begin{remark} \label{rem:m=1}
If $m=1$, the base space $B=\Conf(\R^d,m)=\R^d$ of the Fadell-Neuwirth bundle is contractible, and the bundle is trivial. The parametrized topological complexity of this trivial bundle is equal to the (standard) topological complexity of the fiber $X=\Conf(\R^d\smallsetminus\sO_1,n)$, see \cite[Ex.~4.2]{CFW}, and Theorem \ref{thm:main} is a restatement of results of \cite{FG} in this instance, since $X$ is homotopy equivalent to $\Conf(\R^d,n+1)$.
\end{remark}

We subsequently assume that $m\ge 2$. 
We first show that 
\begin{equation}  \label{eq:lower}
{\tc}\bigl[p\colon \Conf(\R^d,m+n) \to \Conf(\R^d,m)\bigr] \ge 2n+m-2.
\end{equation}
By Proposition \ref{prop:cup}, this is a consequence of the following. 

\begin{prop} \label{prop:CL2}
For $d\ge 2$ even, $E=\Conf(\R^d,m+n)$ and $B=\Conf(\R^d,m)$, the ideal 
\[
\ker[\Delta^* \colon H^*(E\times_BE) \to H^*(E)]
\] 
in $H^*(E\times_BE)$ has cup length $\cl(\ker \Delta^*)\ge 2n+m-2$.
\end{prop}
\begin{proof}
The ideal 
\begin{equation} \label{eq:kerideal}
{\mathcal J} = \langle \omega^{}_{i,j}-\omega'_{i,j} \mid 1\le i<j\ \text{and}\ m<j\le n+m\rangle
\end{equation} 
is generated by degree $d-1$ elements in $H^*(E\times_BE)$. 
One can check (cf. \cite[Prop.~9.4]{CFW}) that ${\mathcal J} \subseteq \ker\Delta^*$. So to prove the proposition it is enough to show that $\cl({\mathcal J}) \ge 2n+m-2$.
We establish this by showing that the product
\begin{equation*} \label{eq:not0even}
\Psi=
\prod_{i=1}^{m} (\omega^{}_{i,m+1}-\omega'_{i,m+1}) 
\prod_{j=m+2}^{m+n} (\omega^{}_{1,j}-\omega'_{1,j})
\prod_{j=m+2}^{m+n} (\omega^{}_{j-1,j}-\omega'_{j-1,j})
\end{equation*}
is nonzero in $H^*(E\times_BE)$. 

If $a_i,b_i$, $1\le i \le q$, are cohomology classes of the same degree, then
\begin{align*}
\prod_{i=1}^\ell (a_i-b_i) &= \sum_{S \subset [q]} (-1)^{|S|} c_1c_2\cdots c_q, \ 
\text{where}\ 
c_j=\begin{cases} a_j&\text{if $j \notin S$,} \\ b_j &\text{if $j \in S$.}\end{cases}\\
\intertext{Using this, we have}
\prod_{i=1}^{m} (\omega^{}_{i,m+1}-\omega'_{i,m+1}) &=
\sum_{S} (-1)^{|S|} \lambda_{1} \cdots \lambda_{m}, \hskip 46pt
\lambda_i=\begin{cases}
\omega^{}_{i,m+1}&\text{if $i\notin S$,}\\
\omega'_{i,m+1}&\text{if $i\in S$,}
\end{cases}\\
\prod_{j=m+2}^{m+n} (\omega^{}_{1,j}-\omega'_{1,j})&=
\sum_{T_1} (-1)^{|T_1|} \mu_{m+2} \cdots \mu_{m+n},\hskip 14pt
\mu_j=\begin{cases}
\omega^{}_{1,j}&\text{\hskip 10pt if $j\notin T_1$,}\\
\omega'_{1,j}&\text{\hskip 10pt if $j\in T_1$,}
\end{cases}\\
\prod_{j=m+2}^{m+n} (\omega^{}_{j-1,j}-\omega'_{j-1,j}) &= 
\sum_{T_2}(-1)^{|T_2|}\xi_{m+2}\cdots\xi_{m+n},\hskip 18pt
\xi_j=\begin{cases}
\omega^{}_{j-1j}&\text{if $j\notin T_2$,}\\
\omega'_{j-1j}&\text{if $j\in T_2$,}
\end{cases}\end{align*}
where 
$S\subset[m]$ and 
$T_1,T_2\subset[m+2,m+n]=\{m+2,m+3,\dots,m+n\}$ 
are increasing sequences.

For $T=(j_1,\dots,j_\ell)$ an increasing sequence in $[p,q]$,  denote the complementary sequence by  $T^{\cc}=(p,\dots,\widehat{j_1},\dots,\widehat{j_\ell},\dots,q)$, and let $\epsilon_T$ be the sign of the shuffle permutation taking $[p,q]$ to $(T^{\cc},T)$. Denote the constant sequence $(1,1,\dots,1)$ (of appropriate length) by $\mathbf{1}$, and let $T-{\mathbf{1}}=(j_1-1,\dots,j_\ell-1)$. 
Recalling that the cohomology classes $\omega^{}_{i,j}$ and $\omega'_{i,j}$ are of odd degree $d-1$, 
the latter two products above may be expressed as
\begin{equation} \label{eq:2of3prods}
\begin{aligned}
\prod_{j=m+2}^{m+n} (\omega^{}_{1,j}-\omega'_{1,j})&=
\sum_{T_1}(-1)^{|T_1|}\epsilon_{T_1} \omega^{}_{\mathbf{1},T_1^{\cc}} \omega'_{{\mathbf{1}},T_1},\\
\prod_{j=m+2}^{m+n} (\omega^{}_{j-1,j}-\omega'_{j-1,j}) &= 
\sum_{T_2}(-1)^{|T_2|}\epsilon_{T_2} \omega^{}_{T_2^{\mathsf c}-{\mathbf{1}},T_2^{\cc}} \omega'_{T_2-{\mathbf{1}},T_2}.
\end{aligned}
\end{equation}
Since, for $i=1,2$, $T_i$ and $T_i^\cc$ are increasing sequences in $[m+2,m+n]$ and $\b1<T_i$, $\b1<T_i^\cc$, and $T_i-\b1<T_i$, the monomials $\omega^{}_{\mathbf{1},T_1^{\cc}} \omega'_{{\mathbf{1}},T_1}$ and 
$ \omega^{}_{T_2^{\mathsf c}-{\mathbf{1}},T_2^{\cc}} \omega'_{T_2-{\mathbf{1}},T_2}$ arising in \eqref{eq:2of3prods} are elements of the basis for $H^*(E\times_BE)$ of Proposition \ref{prop:basis}.

Similarly, with $R=(m+1,m+1,\dots,m+1)$,  the first of the three products above may be expressed as
\begin{equation*} 
\begin{aligned}
\prod_{i=1}^{m} (\omega^{}_{i,m+1}-\omega'_{i,m+1}) &=
\sum_{S} (-1)^{|S|} \epsilon_S  \omega^{}_{S^\cc,R} \omega'_{S,R}\\
&=\sum_{\emptyset \subsetneq S \subsetneq [m]} (-1)^{|S|} \epsilon_S  \omega^{}_{S^\cc,R} \omega'_{S,R}
+\epsilon_\emptyset \omega^{}_{[m],R}+(-1)^m \epsilon_{[m]} \omega'_{[m],R}.
\end{aligned}
\end{equation*}
If $|S|\ge 2$ or $|S^\cc|\ge 2$, the monomial $\omega^{}_{S^\cc,R} \omega'_{S,R}$ is not an element of the basis of Proposition \ref{prop:basis}. Rewriting using Proposition \ref{prop:rewrite} and some sign simplification yields
\begin{equation} \label{eq:first}
\begin{aligned}
\prod_{i=1}^{m} (\omega^{}_{i,m+1}-\omega'_{i,m+1}) &=
\sum_{\emptyset \subsetneq S \subsetneq [m]}  
(-1)^{|S^\cc|}
\epsilon_S  
\Biggl[ \sum_{I_1} (-1)^{d_{I_1}} \omega^{}_{I_1,K_1}\Biggr]
\Biggl[ \sum_{I_2} (-1)^{d_{I_2}} \omega'_{I_2,K_2}\Biggr]\\
&\qquad + (-1)^m\epsilon_\emptyset \sum_{I_1} (-1)^{d_{I_1}} \omega^{}_{I_1,K_1}
+ \epsilon_{[m]}  \sum_{I_2} (-1)^{d_{I_2}} \omega'_{I_2,K_2}
\end{aligned}
\end{equation}
where $I_1$ and $I_2$ range over all $S^\cc$- and $S$-admissible sequences respectively, and if $S^\cc=(i_1,\dots,i_p)$ and $S=(j_1,,\dots,j_q)$, then $K_1=(i_2,\dots,i_{p-1},m+1)$ and $K_2=(j_2,\dots,j_{q-1},m+1)$. 
If, for instance, $S=(i)$ has cardinality one, then $I_2=(i)$ is the only $S$-admissible sequence, $K_2=(m+1)$, and the relevant sum in \eqref{eq:first} above consists of the single term $-\omega'_{i,m+1}$. 
Note also that $K_1=[2,m+1]$ if $S=\emptyset$ and $S^\cc=[m]$, while 
$K_2=[2,m+1]$ if $S=[m]$ and $S^\cc=\emptyset$.

The product $\Psi$ 
may then be obtained by multiplying the expressions of \eqref{eq:first} and \eqref{eq:2of3prods}. Expanding yields an expression of $\Psi$ as a linear combination of monomials $\omega^{}_{P_1,Q_1}\omega^{}_{P_2,Q_2}\omega'_{P_3,Q_3}$, where $Q_1\subset[m]$ and $Q_2,Q_3 \subset [m+1,m+n]$. Some of these monomials are elements of the basis of Proposition \ref{prop:basis}, others are not. One of the basis elements appearing in this expansion of $\Psi$ is
\begin{equation} \label{eq:theone}
\omega^{}_{1,2}\omega^{}_{1,3}\cdots\omega^{}_{1,m}\omega^{}_{1,m+1}\omega^{}_{1,m+2}\cdots\omega^{}_{1,m+n}
\omega'_{m+1,m+2}\omega'_{m+2,m+3}\cdots\omega'_{m+n-1,m+n}.
\end{equation}
This element 
is obtained by taking $S=\emptyset$, $S^\cc=[m]$ and $I_1=\b1$ in \eqref{eq:first}, so that the expansion of $\omega'_{S,R}$ is simply $1$, and by taking $T_1=T_2^\cc=\emptyset$, $T_1^\cc=T_2=[m+2,m+n]$ in \eqref{eq:2of3prods}, so that $\omega'_{\b1,T_1}=\omega^{}_{T_2^\cc-\b1,T_2^\cc}=1$. 
It 
may be expressed briefly as $x=\omega^{}_{\b1,K}\omega^{}_{\b1,T_1^\cc}\omega'_{T_2-\b1,T_2}$, where $K=[2,m+1]$.

We assert that the basis element $x=\omega^{}_{\b1,K}\omega^{}_{\b1,T_1^\cc}\omega'_{T_2-\b1,T_2}$ is unaffected by rewriting non-basis monomials arising in the expansion of $\Psi$ using the three term relations. This will insure the non-vanishing of $\Psi$ as needed.  
We will establish this assertion by (sketching) an elementary, albeit delicate, analysis of the affect of the three term relations on monomials  in the expansion of $\Psi$.

Let $y \neq x$ 
be a monomial in the expansion of $\Psi$. From the expansions \eqref{eq:first} and \eqref{eq:2of3prods}, 
up to sign, 
we have
\begin{equation} \label{eq:almost}
y=
\bigl(\omega^{}_{I_1,K_1}\omega^{}_{\b1,T_1^\cc}\omega^{}_{T_2^\cc-\b1,T_2^\cc}\bigr)
\bigl(\omega'_{I_2,K_2}\omega'_{\b1,T_1^{}}\omega'_{T_2^{}-\b1,T_2^{}}\bigr),
\end{equation}
where
$I_1$ and $I_2$ are (if non-empty) $S^\cc$- and $S$ admissible sequences for some $S\subset [m]$, and 
 for $j=1,2$, $K_j$ is either empty or is an increasing sequence in $[2,m+1]$ as described following \eqref{eq:first},  
and $T_j$ and $T_j^\cc$ are complementary increasing sequences in $[m+2,m+n]$. From Proposition \ref{prop:rewrite}, non-empty sequences $K_1$ and $K_2$ are of the form $(k_1,\dots,k_\ell,m+1)$ with  $k_\ell\le m$, so $\widehat{K}_1$ and $\widehat{K}_2$ are increasing sequences in $[2,m]$ of the form $(k_1,\dots,k_\ell)$, where $K_j=(\widehat{K}_j,m+1)$. Since $\omega'_{i,j}=\omega^{}_{i,j}$ for $1\le i<j\le m$, up to sign, the monomial 
$y$ 
can be rewritten as
\begin{equation} \label{eq:almost2}
y=\bigl(\omega^{}_{\widehat{I}_1,\widehat{K}_1}\omega^{}_{\widehat{I}_2,\widehat{K}_2} \bigr)
\bigl(\omega^{}_{\alpha,m+1}\omega^{}_{\b1,T_1^\cc}\omega^{}_{T_2^\cc-\b1,T_2^\cc}\bigr)
\bigl(\omega'_{\beta,m+1}\omega'_{\b1,T_1^{}}\omega'_{T_2^{}-\b1,T_2^{}}\bigr),
\end{equation}
where $I^{}_1=(\widehat{I}_1,\alpha)$ and $I^{}_2=(\widehat{I}_2,\beta)$.

Suppose the monomial $y$ of \eqref{eq:almost} is not an element of the basis of Proposition \ref{prop:basis}. 
First, consider the case where the subset $S$ of $[m]$ in \eqref{eq:first} is non-empty, so that $K_2\neq \emptyset$. As indicated in \eqref{eq:almost2} above, this gives rise to a factor of $\omega'_{\beta,m+1}$ in the monomial $y$. Subsequent simplifications, for instance if 
$\omega'_{\b1,T_1^{}}\omega'_{T_2^{}-\b1,T_2^{}}$ is not a basis element, either annihilate $y$ or give rise to basis elements involving $\omega'_{\beta,m+1}$ or $\omega'_{1,m+1}$. No factor of this form appears in the monomial $x$ of \eqref{eq:theone}.

It remains to consider the case where the subset  $S$ of $[m]$ in \eqref{eq:first} is empty. For $S=\emptyset$, we have $K_1=[2,m+1]$ and $\alpha \le m$ in \eqref{eq:almost2}. In this instance, 
\begin{equation} \label{eq:almost3}
y=\bigl(\omega^{}_{\widehat{I}_1,\widehat{K}_1}\bigr)
\bigl(\omega^{}_{\alpha,m+1}\omega^{}_{\b1,T_1^\cc}\omega^{}_{T_2^\cc-\b1,T_2^\cc}\bigr)
\bigl(\omega'_{\b1,T_1^{}}\omega'_{T_2^{}-\b1,T_2^{}}\bigr),
\end{equation}
where $I^{}_1=(\widehat{I}_1,\alpha)$ and $\widehat{K}_1=[2,m]$. We have either $T_1\neq\emptyset$ or $T^\cc_2\neq\emptyset$, since the basis element $x$ of \eqref{eq:theone} is obtained by taking $S=\emptyset$ and $T_1=T^\cc_2=\emptyset$.

If $T_1\neq \emptyset$, then $\omega'_{1,k}$ is a factor of $y$, where $k\in[m+2,m+n]$ denotes the largest element of $T_1$. 
This factor survives in each term of the expansion of $y$ arising from application of the three term relations to $\omega^{}_{\alpha,m+1}\omega^{}_{\b1,T_1^\cc}\omega^{}_{T_2^\cc-\b1,T_2^\cc}$ and resulting expressions. Consider the affect of expanding the factor $\omega'_{\b1,T_1^{}}\omega'_{T_2^{}-\b1,T_2^{}}$ using the three term relations in each such expression. Rewriting is required only when $T_1\cap T_2\neq \emptyset$, and involves relations $\omega'_{1,q}\omega'_{q-1,q}=\omega'_{1,q-1}(\omega'_{q-1,q}-\omega'_{1,q})$ for $q\in T_1\cap T_2$. If $k \notin T_2$, the factor $\omega'_{1,k}$ is unaffected. 
If, on the other hand, $k \in T_2$, then rewriting $\omega'_{1,k}\omega'_{k-1,k}$ yields expressions involving $\omega'_{1,k-1}$. Continuing as necessary yields a linear combination of basis elements, each of which contains a factor of $\omega'_{1,j}$, for some $j$, $m+1\le j \le k$. No factor of this form appears in the monomial $x$ of \eqref{eq:theone}.

Finally, if $T_1=\emptyset$, then $T_2^\cc\neq\emptyset$. This implies that $T_2$ is a proper subset of $[m+2,m+n]$, and consequently that the factor
$\omega'_{m+1,m+2}\omega'_{m+2,m+3}\cdots\omega'_{m+n-1,m+n}$ appearing in  the monomial $x$ of \eqref{eq:theone} cannot appear in $y$. Since $\omega'_{T_2-\b1,T_2}$ is a basis element, any necessary expansion of $y$ involves applications of the three term relations to the factor 
$\omega^{}_{\b1,T_1^\cc}\omega^{}_{T_2^\cc-\b1,T_2^\cc}$. Since these, and subsequent simplifications, cannot introduce any factors of the form $\omega'_{p,q}$, the factor
$\omega'_{m+1,m+2}\omega'_{m+2,m+3}\cdots\omega'_{m+n-1,m+n}$ of $x$ cannot appear in any resulting monomial.

Thus, as asserted, expressing $\Psi$ in terms of the basis of Proposition \ref{prop:basis} does not alter the summand  $x=\omega^{}_{\b1,K}\omega^{}_{\b1,T_1^\cc}\omega'_{T_2-\b1,T_2}$. Therefore, $\Psi \neq 0$ and $\cl(\ker \Delta^*)\ge \cl({\mathcal J})\ge 2n+m-2$ as required.
\end{proof}

As noted above, the inequality \eqref{eq:lower} follows from 
Proposition \ref{prop:CL2}.  
We establish the reverse inequality for the case $d=2$ of the plane and for the case $d\ge 4$ of higher even dimensions using different methods. 
Since the result in the planar case will play a role in the proof in the higher dimensional case, we begin with the former.

\subsection*{The plane}
Consider the case $d= 2$ of the plane $\R^2=\C$. Express the configuration space $\Conf(\R^2,\ell)$ as
\[
\Conf(\R^2,\ell) = \Conf(\C,\ell) = \{(y_1,\dots,y_\ell) \in \C^\ell \mid y_i \neq y_j\ \text{if}\ i \neq j\}
\]
in complex coordinates. 

For any $\ell\ge 3$, the map $h_\ell\colon \Conf(\C,\ell) \to \Conf(\C\smallsetminus\{0,1\},\ell-2) \times \Conf(\C,2)$ defined by 
\begin{equation} \label{eq:homeo}
h_{\ell}(y_1,y_2,y_3,\dots,y_\ell) =\left(\Bigl(\frac{y_3-y_1}{y_2-y_1},\dots,\frac{y_\ell-y_1}{y_2-y_1}\Bigr),\bigl(y_1,y_2\bigr)\right) 
\end{equation}
is a homeomorphism. It follows that the bundle $p\colon \Conf(\C,m+n) \to \Conf(\C,m)$ is trivial for $m=2$. The parametrized topological complexity is then equal to the 
topological complexity of the fiber $\Conf(\C\smallsetminus\{0,1\},n)$, see \cite[Ex.~4.2]{CFW}. 
Since $\tc(\Conf(\C\smallsetminus\{0,1\},n))=2n$ as shown in \cite{FGY}, for $m=2$, we have 
\[
\tc[p\colon \Conf(\C,n+2) \to \Conf(\C,2)]=\tc(\Conf(\C\smallsetminus\{0,1\},n))=2n
\]
as asserted.

For $m\ge 3$, the maps \eqref{eq:homeo} give rise to an equivalence of fibrations
\[
\begin{CD}
\Conf(\C,m+n) @>{h_{m+n}}>> \Conf(\C\smallsetminus\{0,1\},m+n-2) \times \Conf(\C,2) \\
@V{p}VV @VV{q}V \\
\Conf(\C,m) @>{h_m}>> \Conf(\C\smallsetminus\{0,1\},m-2) \times \Conf(\C,2),
\end{CD}
\]
where $q=q'\times q''$, with $q'$ the 
projection onto the first $m-2$ coordinates 
and $q''=\id$ the identity map. Since $\tc[q''\colon \Conf(\C,2)\to \Conf(\C,2)]=0$, the 
product inequality Proposition \ref{prop:product} implies that ${\tc}\left[p\colon \Conf(\C,m+n) \to \Conf(\C,m)\right]$ is less than or equal to 
\begin{equation} \label{eq:qprime}
\tc\left[q'\colon \Conf(\C\smallsetminus\{0,1\},m+n-2)
\to  \Conf(\C\smallsetminus\{0,1\},m-2)\right].
\end{equation}
Let $E'=\Conf(\C\smallsetminus\{0,1\},m+n-2)$ and $B'=\Conf(\C\smallsetminus\{0,1\},m-2)$. The 
fiber of $q'\colon E' \to B'$ is the configuration space $X=\Conf(\C\smallsetminus\sO_m,n)$, which has the homotopy type of a CW-complex of dimension $n$. Similarly, $B'$ has the homotopy type of a CW-complex of dimension $m-2$. Using Proposition \ref{prop:upper}, 
we obtain the following upper bound for \eqref{eq:qprime}: 
\[
\tc[q'\colon E' \to B'] \le 2\dim(X)+\dim(B)=2n+m-2.
\]
Combining the above observations yields
\[
{\tc}\left[p\colon \Conf(\C,m+n) \to \Conf(\C,m)\right] \le 2n+m-2.
\]
Together with the lower bound \eqref{eq:lower}, this completes the proof of Theorem \ref{thm:main} in the case $d=2$ of the plane $\R^2=\C$.

Theorem \ref{thm:main} for the planar case $d=2$ 
informs on the structure of the cohomology ring $H^*(E\times_BE)$ for any even $d$. 
This structure will be utilized in the case $d\ge 4$ of higher even dimensions below.
Recall the ideal ${\mathcal J}$ in $H^*(E\times_BE)$ from \eqref{eq:kerideal}.
\begin{cor} \label{cor:clJ}
For positive integers $m$ and $d$ with $m\ge 2$ and $d\ge 2$ even, let 
$E=\Conf(\R^d,m+n)$, and $B=\Conf(\R^d,m)$. Then 
the ideal 
\[
{\mathcal J}=  \langle \omega^{}_{i,j}-\omega'_{i,j} \mid 1\le i<j\ \text{and}\ m<j\le n+m\rangle
\] 
in $H^*(E\times_BE)$ has cup length $\cl({\mathcal J})= 2n+m-2$.
\end{cor}

\subsection*{Higher even dimensions} \label{sec:high dim}
For $d\ge 4$ even, we use obstruction theory to complete the proof of Theorem \ref{thm:main}.

The Schwarz genus of a fibration $p\colon E \to B$ with fiber $X$ 
is at most $r-1$ if and only if its $r$-fold fiberwise join admits a continuous section, cf. \cite[Thm.~3]{Sch}. Consequently, 
$\tc[p\colon E \to B]\le r-1$ if and only if the $r$-fold fiberwise join
\[
\Pi_r\colon \Jo_r (E^I_B) \to E\times_B E
\] 
admits a section. 
Thus, to show that 
${\tc}\left[p\colon \Conf(\R^d,m+n)\to \Conf(\R^d,m)\right] \le 2n+m-2$, it suffices to prove the following.

\begin{prop}  \label{prop:join}
For positive integers $m$ and $d$ with $m\ge 2$ and $d\ge 4$ even, let 
$E=\Conf(\R^d,m+n)$, $B=\Conf(\R^d,m)$,  and $r=2n+m-1$. 
Then the fibration 
\[
\Pi_r \colon \Jo_r E^I_B \to E\times_B E
\] 
admits a section.
\end{prop}
\begin{proof}
For any fibration $p\colon E \to B$, the fiber of $\Pi_r$ is $\Jo_r (\Omega{X})$, the $r$-fold join of the loop space of $X$. In the case of the Fadell-Neuwirth bundle of configuration spaces, we have $B=\Conf(\R^d,m)$, $E=\Conf(\R^d,m+n)$, and $X=\Conf(\R^d\smallsetminus\sO_m,n)$. As noted previously, $X$ is $(d-2)$-connected. Since the join of $p$- and $q$-connected CW-complexes is $(p+q+2)$-connected, the fiber $\Jo_r (\Omega X)$ of $\Pi_r$ is $(rd-r-2)$-connected.

From the connectivity of $X$, 
the primary obstruction to the existence of a section of $\Pi_r \colon \Jo_r E^I_B \to E\times_B E$ is an element 
$\Theta_r \in H^{r(d-1)}(E\times_B E; \pi_{r(d-1)-1}(\Jo_r \Omega X))$. 
Since $\hdim \bigl(E\times_B E\bigr)=r(d-1)$ as noted in  \eqref{eq:hdims}, higher obstructions vanish for dimensional reasons. So
$\Theta_r$  is the only obstruction.  
By the Hurewicz theorem, we have $\pi_{r(d-1)-1}(\Jo_r \Omega X)=H_{r(d-1)-1}(\Jo_r\Omega X)$. 
For spaces $Y$ and $Z$ with 
torsion free integral homology, the (reduced) homology of the join is given by $\widetilde{H}_{q+1}(Y\Jo Z)=\bigoplus_{i+j=q} \widetilde{H}_i(Y) \otimes \widetilde{H}_j(Z)$. This, together with the fact that the homology groups of 
$X$ (and $\Omega X$) are free abelian, yields
\[
\pi_{r(d-1)-1}(\Jo_r \Omega X)=H_{r(d-1)-1}(\Jo_r\Omega X) = [\widetilde{H}_{d-1}(X)]^{\otimes r}
= [H_{d-1}(X)]^{\otimes r},
\]
the last equality since $d\ge 4$. Thus, $\Theta_r \in H^{r(d-1)}(E\times_B E;  [H_{d-1}(X)]^{\otimes r})$.

By \cite[Thm.~1]{Sch}, the obstruction $\Theta_r$ decomposes as $\Theta_r=\theta\smile \dots \smile \theta = \theta^r$, where $\theta \in H^{d-1}(E\times_B E;  H_{d-1}(X))$ is the primary obstruction to the existence of a section of $\Pi\colon E^I_B \to E\times_B E$.  
Since $E\times_B E$ is simply connected,  the system of coefficients $H_{d-1}(X)$ on $E\times_B E$ is trivial. As noted above,  $H_{d-1}(X)$ is torsion free. By Proposition \ref{prop:HEBE}, the cohomology ring $H^*(E\times_B E)$ is also torsion free. It follows that $H^*(E\times_B E;[H_{d-1}(X)]^{\otimes q})$ is  torsion free for any $q\ge 1$.

Since $\theta$ is the primary obstruction to the existence of a section the fibration $\Pi\colon E^I_B \to E\times_B E$, 
we have
\[
\theta \in \ker[\Delta^*\colon H^*(E\times_B E;H_{d-1}(X)) \longrightarrow H^*(E;\Delta^*H_{d-1}(X))].
\]
For brevity, denote the free abelian group $H_{d-1}(X)$ by $A$. Using a Universal Coefficient theorem (for a (co)chain complex computing $H^*(E\times_B E)$), we can identify $H^{d-1}(E\times_B E;A)$ with $H^{d-1}(E\times_B E)\otimes A$, and $H^{d-1}(E;A)$ with $H^{d-1}(E)\otimes A$. 
With these identifications,  we have $\Delta^*\colon H^{d-1}(E\times_B E)\otimes A \to H^{d-1}(E)\otimes A$, and $\theta\in\ker(\Delta^*)$ may be expressed as a linear combination of elements of the form $\eta_j \otimes a_j$, where the elements $\eta_j$ are the degree $d-1$ generators of 
$\ker[\Delta^*\colon H^*(E\times_B E)) \to H^*(E;\Z))]$ and $a_j \in A$.

The $r$-fold cup product $\Theta_r=\theta^r \in H^{r(d-1)}(E\times_B E)\otimes A^{\otimes r}$ is then realized as a linear combination of elements of the form $\eta_J \otimes a_J$, where $\eta_J=\eta_{j_1}\smile \dots \smile \eta_{j_r}$ is an $r$-fold cup product of degree $d-1$ generators of $\ker[\Delta^*\colon H^*(E\times_B E)) \to H^*(E))]$, and $a_J \in A^{\otimes r}$. But the degree $d-1$ generators of $\ker\Delta^*$ are the generators of the ideal $\mathcal J$ of \eqref{eq:kerideal}. 
As noted in Corollary \ref{cor:clJ}, we have $\cl({\mathcal J}) = 2n+m-2$. It follows that for $r=2n+m-1$, we have ${\mathcal J}^r=0$, and consequently $\theta^r=0$. Since the primary obstruction $\Theta_r=\theta^r$ vanishes, 
the fibration $\Pi_r \colon \Jo_r E^I_B \to E\times_B E$ admits a section.
\end{proof}
This
completes the proof of Theorem \ref{thm:main} in the case where $d\ge 4$ is even.

\begin{ack}
The first author thanks Emanuele Delucci, Nick Proudfoot, 
and He Xiaoyi  for productive conversations, and the organizers of the virtual workshop \emph{Arrangements at Home} for facilitating several of these conversations. Portions of this work were undertaken when the first and second authors visited the University of Florida Department of Mathematics in November, 2019. We thank the department for its hospitality and for providing a productive mathematical environment. We also thank the anonymous referees for their helpful comments.
\end{ack}

\section*{Appendix} \label{sec:appendix}

In this appendix, we state and prove a general result which includes as a special case the fact noted in Proposition \ref{prop:Pi} that, for a fibration $p\colon E \to B$,  the map $\Pi \colon E^I_B \to E\times_B E$ is also a fibration.

Let $p: E\to B$ be a Hurewicz fibration. For a topological space $X$, let $E^X_B$ denote the space of all continuous maps $f\colon X\to E$ lying in a single fiber of $p$, that is,  such that the composition $p\circ f \colon X\to B$ is a constant map. Equip the space $E^X_B$ with the compact-open topology. The following result may be compared with \cite[Thm. 2.8.2]{Spa}

\begin{propo*}
Let $(K, L)$ be a pair consisting of a finite CW-complex and a subcomplex. Then the restriction map 
\[
\Pi\colon E^K_B \to E^L_B, \quad \Pi(f) = f|L \quad \mbox{for}\quad f\in E^K_B, 
\]\
is a Hurewicz fibration.
\end{propo*}
\begin{proof}
 Let $\lambda: \overline B \to E^I$ be a lifting function for the fibration $p\colon E \to B$ (see \cite[Sec. 2.7]{Spa}). 
Here 
$\overline B=\{(e, \omega)\in E\times B^I \mid p(e)=\omega(0)\}$ and $\lambda(e, \omega)\in E^I$ satisfies 
$p\circ \lambda(e, \omega)=\omega$ and $\lambda(e, \omega)(0)=e$. Our goal is to construct a lifting function 
\begin{equation}\label{eq:Lambda} \tag{$\dag$}
\Lambda: \overline{E^L_B}\, \to\, \left(E^K_B\right)^I
\end{equation}
for $\Pi$. 
Here $\overline{E^L_B}$ is the set of pairs 
\[
\{(F, G)\in E^K_B \times \left(E^L_B\right)^I \mid F(x)=G(x, 0) \, \, \mbox{for}\, x\in L\}.
\]
Clearly, $G\in (E^L_B)^I$ can be viewed as a map $G\colon L\times I\to E$ such that 
for any $t\in I$ the image $G(L\times t)\subset E$ lies in a single fiber. 
The map $F\in E^K_B$ satisfies $F(x) = G(x, 0)$ for $x\in L$. 
Denote by $\omega(t) = p(G(x, t))$ (where $x\in L$) the  path in $B$ obtained by applying the projection $p$. 
For $x\in K$, the formula $\tilde F(x, t) = \lambda(F(x), \omega)(t)$ defines a map $\tilde F\in (E^K_B)^I$ satisfying $\tilde F(x, 0) =F(x)$ and $p(\tilde F(x, t))=\omega(t)= p(G(x, t))$. However, we may not guarantee that the condition $\tilde F(x, t)= G(x, t)$ for $x\in L$ and $t\in I$ holds. 

Let $\omega^{[\tau, 1]}$ denote the path $s\mapsto \omega^{[\tau, 1]}(s) = \omega(\tau +(1-\tau) s)$, where $s\in [0,1]$. Define 
$\alpha_G\colon L\times I\times I\to E$
by
\[
\alpha_G(x, \tau, t) = \begin{cases}
G(x, t),& \text{for $0\le t\le \tau$,}\\ 
\lambda(G(x, \tau), \omega^{[\tau, 1]})\left(\frac{t-\tau}{1-\tau}\right), & \text{for $ \tau\le t\le 1$.}
\end{cases}
\]
For $x\in L$, we have $p(\alpha_G(x, \tau, t))=\omega(t)$, $\alpha_G(x, 0, t) = \lambda(G(x, 0), \omega)(t)=\tilde F(x, t)$, and
$\alpha_G(x, 1, t)=G(x, t)$. 
Thus $\alpha_G$ is a fibrewise homotopy between $\Pi(\tilde F)$ and 
$G$ in $(E^L_B)^I$. 

Let $\rho: K\times I\to K\times 0\cup L\times I$ be a retraction. 
Define $H \colon K\times I\times I \to E$ to be the composition 
Denote by $H$ the composition
\[
K\times I\times I \xrightarrow{\rho\times1}  K\times 0\times I\cup L\times I\times I \xrightarrow{\tilde F\ \cup \ \alpha_G}\ E.
\]
The map $h(x, t) = H(x, 1, t)$ is an element of $(E^K_B)^I$ satisfying
$h(x,0)= F(x)$ and $h(x, t) =G(x, t)$ for $ x\in L$.
Hence we may define the lifting function \eqref{eq:Lambda} by setting
$\Lambda(F, G)=h$.
\end{proof}

Proposition \ref{prop:Pi}, asserting that $\Pi\colon E^I_B \to E\times_BE$, $\gamma \mapsto (\gamma(0),\gamma(1))$, is a fibration, may be obtained by taking  $K=[0,1]$ and $L=\{0,1\}$ in the above result.

\newcommand{\arxiv}[1]{{\texttt{\href{http://arxiv.org/abs/#1}{{arXiv:#1}}}}}

\newcommand{\MRh}[1]{\href{http://www.ams.org/mathscinet-getitem?mr=#1}{MR#1}}

\bibliographystyle{amsplain}

\begin{thebibliography}{99}

\bibitem{Adi} K. Adiprasito, \emph{Combinatorial stratifications and minimality of 2-arrangements}. J. Topol. \textbf{7} (2014), 1200--1220; \MRh{3286901}.

\bibitem{dc18} D. Cohen, \emph{Topological complexity of classical configuration spaces and related objects}, in: \emph{Topological complexity and related topics}, pp. 41--60, 
Contemp. Math., vol. 702, Amer. Math. Soc., Providence, RI, 2018; \MRh{3762831}

\bibitem{CFW} D. Cohen, M. Farber, S. Weinberger, \emph{Topology of parametrized motion planning algorithms}, SIAM J. Appl. Algebra Geom. \textbf{5} (2021), 229--249; \MRh{4272901}.

\bibitem{dLS} M. de Longueville, C. Schultz, \emph{The cohomology rings of complements of subspace arrangements}, Math. Ann. \textbf{319} (2001), 625--646; \MRh{1825401}.

\bibitem{Dr} A. Dranishnikov, \emph{On topological complexity of hyperbolic groups}, 
Proc. Amer. Math. Soc. \textbf{148} (2020), 4547--4556; \MRh{4135318}.

\bibitem{FH} E. Fadell, S. Husseini, \emph{Geometry and Topology of Configuration Spaces}, Springer Monographs in Mathematics, Springer-Verlag, Berlin, 2001; \MRh{1802664}.

\bibitem{Fa03} M. Farber, {\em Topological complexity of motion planning}, 
Discrete Comput. Geom. \textbf{29} (2003), 211--221; \MRh{1957228}.

\bibitem{Fa05} M. Farber,  
{\em{Topology of robot motion planning}}, 
in: \emph{Morse Theoretic Methods in Non-linear 
Analysis and in Symplectic Topology}, pp. 185--230, 
NATO Science Series II: Mathematics, Physics and Chemistry, vol. 217, 
Springer, 2006; 
\MRh{2276952}.

\bibitem{Fa18} M. Farber, \emph{Configuration spaces and robot motion planning algorithms}, 
in: \emph{Combinatorial and Toric Homotopy: Introductory Lectures}, pp. 263--303, 
Lect. Notes Ser. Inst. Math. Sci. Natl. Univ. Singap., vol. 35, World Sci. Publ., Hackensack, NJ, 2018;  
\MRh{3792464}. 

\bibitem{FG} M. Farber, M. Grant, 
\emph{Topological complexity of configuration spaces}, Proc. Amer. Math. Soc. \textbf{137} (2009), 1841--1847; \MRh{2470845}.

\bibitem{FGY} M. Farber, M. Grant, S. Yuzvinsky, 
\emph{Topological complexity of collision free motion planning algorithms in the presence of multiple moving obstacles}, in: \emph{Topology and Robotics}, pp. 75--83, Contemp. Math., vol. 438, Amer. Math. Soc., Providence, RI, 2007; \MRh{2359030}.

\bibitem{GM} M. Goresky, R. MacPherson, \emph{Stratified Morse theory}, Ergebnisse der Mathematik und ihrer Grenzgebiete (3) 14, Springer-Verlag, Berlin, 1988; \MRh{0932724}.

\bibitem{GrM} M. Grant, S. Mescher, \emph{Topological complexity of symplectic manifolds}, Math. Z. \textbf{295} (2020), 667--679; \MRh{4100027}.

\bibitem{IG} C. Ipanaque Zapata, J. Gonz\'alez, \emph{Multitasking collision-free optimal motion planning algorithms in Euclidean spaces}, Discrete Math. Algorithms Appl. \textbf{12} (2020), 2050040, 19 pp.; \MRh{4120432}.

\bibitem{Lat} J.-C. Latombe, \emph{Robot Motion Planning}, Springer Science$+$Business Media, New York, 2012. 

\bibitem{Lav} S.M. LaValle, \emph{Planning Algorithms}, Cambridge University Press, Cambridge, 2006; \hfill \\
\MRh{2424564}.

\bibitem{MW} A. Murillo, J. Wu, \emph{Topological complexity of the work map}, J. Topol. Anal. 
\textbf{13} (2021), , 219--238; \MRh{4243078}.

\bibitem{OT}  P.~Orlik, H.~Terao,
{\em Arrangements of hyperplanes}, Grundlehren Math. Wiss.,
vol.~300, Springer-Verlag, New~York-Berlin-Heidelberg, 1992; \MRh{1217488}.

\bibitem{P19} P. Pave{\u{s}}i\'c, \emph{Topological complexity of a map}, Homology Homotopy Appl. \textbf{21} (2019), 107--130; \hfill \\ \MRh{3921612}. 

\bibitem{Sch} A.S. Schwarz, \emph{The genus of a fiber space}, Amer. Math. Sci. Transl., \textbf{55} (1966), 49–140; \hfill \\ \MRh{0154284}.

\bibitem{Sm} S. Smale, \emph{The fundamental theorem of algebra and complexity theory},  
Bull. Amer. Math. Soc. (N.S.) \textbf{4} (1981), 1--36; 
\MRh{0590817}.

\bibitem{Spa} E. Spanier, \emph{Algebraic Topology}, corrected reprint, Springer-Verlag, New York-Berlin, 1981; \MRh{0666554}. 

\end{thebibliography}

\end{document}